\documentclass{amsart}
\usepackage{amsmath}
\usepackage{amsfonts}
\usepackage{amsthm}

\theoremstyle{plain}
\newtheorem{thm}{Theorem}[section]
\newtheorem{lem}[thm]{Lemma}
\newtheorem{cor}[thm]{Corollary}
\newtheorem{que}[thm]{Question}

\theoremstyle{definition}
\newtheorem{defn}[thm]{Definition}

\title{Numbers and the Heights of Their Happiness}
\author{May Mei}
\address{Department of Mathematics \& Computer Science, Denison University, 43023}
\email{meim@denison.edu}
\author{Andrew Read-McFarland}
\address{Department of Mathematics \& Computer Science, Denison University, 43023}
\email{readmc\_a1@denison.edu}
 
\begin{document}

\begin{abstract}
A generalized happy function, $S_{e,b}$ maps a positive integer to the sum of its base $b$ digits raised to the $e^\text{th}$ power. We say that $x$ is a base $b$, $e$ power, height $h$, $u$ attracted number if $h$ is the smallest positive integer so that $S^{h}_{e,b}(x)=u$. Happy numbers are then base 10, 2 power, 1 attracted numbers of any height. Let $\sigma_{h,e,b}(u)$ denote the smallest height $h$, $u$ attracted number for a fixed base $b$ and exponent $e$ and let $g(e)$ denote the smallest number so that every integer can be written as $x_{1}^{e}+x_{2}^{e}+...+x_{g(e)}^{e}$ for some nonnegative integers $x_{1},x_{2},...,x_{g(e)}$. In this paper we prove that if $p_{e,b}$ is the smallest nonnegative integer such that $b^{p_{e,b}}>g(e)$, $\displaystyle d=\left\lceil \frac{g(e)+1}{1-(\frac{b-2}{b-1})^{e}}+e+p_{e,b}\right\rceil$, and $\sigma_{h,e,b}(u)\geq b^{d}$, then $S_{e,b}(\sigma_{h+1,e,b}(u))=\sigma_{h,e,b}(u)$.
\end{abstract}

\maketitle

\section{Introduction}

Let $S_{e,b}$ be the function that maps a positive base $b$ integer to the sum of its digits raised to the $e^\text{th}$ power, where $e$ is a positive integer. That is for $x=\sum_{i=0}^{n-1}a_i b^i$, with $0\leq a_{i}\leq b-1$ for all $i$,
$$S_{e,b}\left(\sum_{i=0}^{n-1} a_i b^i \right)=\sum_{i=0}^{n-1} a_i^e.$$
If $S_{e,b}^{h}(x)=1$ for some integer $h$, then $x$ is said to be an $e$ power, $b$ happy number. In \cite{Guy:2004}, Guy gives the smallest 2 power, 10 happy numbers of heights 0 through 6 and asks if 78999 is the smallest height 7 happy number. Grundman and Teeple answer Guy in \cite{GrundmanTeeple:2003}, giving the smallest 2 power, 10 happy numbers of heights 0-10, and 3 power, 10 happy numbers of heights 0-8. From Grundman and Teeple's work, one can extract an algorithm for finding the smallest happy number of height $h+1$ if the smallest happy number of height $h$ is known.  The main results of this paper are Theorems \ref{L:trailnines} and \ref{L:ninesmap}, which jointly imply that once the smallest height $h+1$, $u$ attracted, base $b$ number is sufficiently large, applying $S_{e,b}$ to that number will yield the smallest height $h$, $u$ attracted, base $b$ number. The results of this paper hold not only for happy numbers (i.e. 1 attracted), but more generally for $u$ attracted numbers. Moreover, our results hold for all bases and exponents.

\begin{defn}
For a fixed base $b$, exponent $e$, and positive integer $u$, we say that a positive integer $x$ is $u$ attracted if $S_{e,b}^{n}(x)=u$ for some nonnegative integer $n$. If $h$ is the smallest nonnegative integer so that $S_{e,b}^{h}(x)=u$ then $x$ is height $h$, $u$ attracted number. (As a convention, $S_{e,b}^0(x)=x$.)
\end{defn}

\begin{defn}
For a fixed base $b$, exponent $e$, positive integer $u$, and nonnegative integer $h$, let $\sigma_{h,e,b}(u)$ denote the smallest height $h$, $u$ attracted number. That is, the smallest positive integer $k$ with the property that $S^{h}_{e,b}(k)=u$ and $S^{n}_{e,b}(k) \neq u$ for $n<h$. Similarly, for positive $h$, let $\tau_{h,e,b}(u)$ denote the second smallest height $h$, $u$ attracted number. That is $S^{h}_{e,b}(l)=u$, $S^{n}_{e,b}(l)\neq u$ for $n<h$, and $\sigma_{h,e,b}(u)<l$.
\end{defn}

Some of the following proofs rely upon knowing the smallest integer $x$ such that for a given $e$, every integer is expressible as the sum of at most $x$ many integers raised to the $e^{\text{th}}$ power. We define $g(e)$ for this purpose.

\begin{defn}
For a fixed positive integer $e$, let $g(e)$ denote the smallest integer such that every nonnegative integer is expressible as $x_{1}^{e}+x_{2}^{e}+...+x_{g(e)}^{e}$ where $x_{1},x_{2},...,x_{g(e)}$ are all nonnegative integers.
\end{defn}

This is the well-known Waring's problem. Many surveys about the history of this problem exist, see for instance \cite{VaughanWooley:2000}.

For the entirety of this paper, we assume that base $b \geq 2$ is an integer, exponent $e \geq 1$ is an integer, height $h$ is a nonnegative integer, attractor $u$ is a positive integer, and that $x$ denotes a positive integer. Additionally, when we say $\lceil x \rceil=y$ we mean that $y$ is the smallest integer such that $y\geq x$, and similarly, if $\lfloor x \rfloor=y$, then $y$ is the largest integer such that $y\leq x$.

\section{Mapping Attracted Numbers}
In this section, we establish in Theorem $\ref{L:willmap}$ a criterion, depending on $g(e)$ that ensures that $S_{e,b}(\sigma_{h+1,e,b}(u))=\sigma_{h,e,b}(u)$ for a fixed base $b$, exponent $e$, height $h$, and integer $u$.
\begin{lem}
\label{L:digitbound}
Fix a base $b$, exponent $e$, and attractor $u$. The smallest positive integer $x$ such that $S_{e,b}(x)=u$ has $n$ digits, where $ \frac{u}{(b-1)^{e}}  \leq n \leq  \frac{u}{(b-1)^{e}}  + g(e)$.
\end{lem}
\begin{proof}
Since the maximum value of the image of each digit under $S_{e,b}$ is $(b-1)^{e}$, $\frac{u}{(b-1)^{e}} $ is a lower bound for the number of digits of $x$.
Let $q$ and $r$ be the quotient and remainder of $u$ divided by $(b-1)^{e}$, respectively, that is $q$ is a nonnegative integer, $0 \leq r<(b-1)^e$, and $u=q(b-1)^{e}+r$. Let $x_{1},...,x_{g(e)}$ be integers such that $x_{1}^{e}+...+x_{g(e)}^{e}=r$. Since $r<(b-1)^{e}$, $x_{1},...,x_{g(e)}<b-1$ and so are valid digits in base $b$. Without loss of generality, $x_{1}\leq x_{2} \leq ... \leq x_{g(e)}$. Let $y$ be the positive integer formed by the digits $x_{1},x_{2},...,x_{g(e)}$ followed by $q$ digits, each of which is $b-1$. Since $x$ is minimal, it follows that $x \leq y$. So $n$, the number of digits of $x$, must be less than or equal to the number of digits of $y$, which is $\lfloor \frac{u}{(b-1)^{e}} \rfloor + g(e)$.
\end{proof}

\begin{thm}
\label{L:willmap}
Fix a base $b$, exponent $e$, positive height $h$, and attractor $u$. If $ \frac{\sigma_{h,e,b}(u)}{(b-1)^{e}} +g(e) \leq  \frac{\tau_{h,e,b}(u)}{(b-1)^{e}} $, then $S_{e,b}(\sigma_{h+1,e,b}(u))=\sigma_{h,e,b}(u)$.
\end{thm}
\begin{proof}
Let $x$ be the smallest integer such that $S_{e,b}(x)=\sigma_{h,e,b}(u)$. Let $z$ be a height $h$, $u$ attracted number that is greater than $\sigma_{h,e,b}(u)$ (recall that $\tau_{h,e,b}$ is the smallest such number) and $y$ any integer such that $S_{e,b}(y)=z$. That is, $y$ is a height $h+1$, $u$ attracted number whose image is not $\sigma_{h,e,b}(u)$. Let $n$ be the number of digits of $x$ and $m$ be the number of digits of $y$. We will show that $x<y$. By Lemma~\ref{L:digitbound}, $n\leq  \frac{\sigma_{h,e,b}(u)}{(b-1)^{e}} +g(e)$ and $ \frac{\tau_{h,e,b}(u)}{(b-1)^{e}}  \leq \frac{z}{(b-1)^e}\leq m$. By hypothesis, $ \frac{\sigma_{h,e,b}(u)}{(b-1)^{e}} +g(e) \leq  \frac{\tau_{h,e,b}(u)}{(b-1)^{e}} $, so $n \leq m$.  If $n<m$, then $x<y$, so let us suppose that $n=m$. It must then be the case that $\frac{\sigma_{h,e,b}(u)}{(b-1)^{e}} +g(e)=\frac{z}{(b-1)^{e}}$. Since $S_{e,b}(y)= z$ and $y$ has $m=\frac{z}{(b-1)^{e}}$ digits, $y$ is the concatenation of $m$ digits, each of which is $b-1$. Since $x \neq y$ (as they have different images under $S_{e,b}$) and $x$ and $y$ have the same number of digits, at least one digit of $x$ is not $b-1$. Thus, $x<y$. Hence $x$ is less than every other height $h+1$, $u$-attracted number, and so $x=\sigma_{h+1,e,b}(u)$. Since $S_{e,b}(x)=\sigma_{h,e,b}(u)$, $S_{e,b}(\sigma_{h+1,e,b}(u))=\sigma_{h,e,b}(u)$.\end{proof}

From \cite{GrundmanTeeple:2003}, it is known that $\sigma_{7,2,10}=78999$ and $\tau_{7,2,10}(1)=79899$. 
\begin{que}
Under what conditions is $\tau_{h,e,b}(u)$ a permutation of the digits of $\sigma_{h,e,b}(u)$?
\end{que}

\section{Large $u$ Attracted Numbers}\label{S:largeuattracted}
In this section, we prove Theorems \ref{L:trailnines} and \ref{L:ninesmap}, which imply that once $\sigma_{h,e,b}(u)$ is sufficiently large, $S_{e,b}(\sigma_{h+1,e,b}(u))=\sigma_{h,e,b}(u)$. 
\begin{thm}
\label{L:trailnines}
Fix a base $b$, exponent $e$, positive height $h$, and attractor $u$. Let $\delta$ be a positive integer, and let $$d=\frac{g(e)+1}{1-(\frac{b-2}{b-1})^{e}}+\delta.$$ If $\sigma_{h,e,b}(u)$ has at least $d$ digits, then the base $b$ expansion of $\sigma_{h,e,b}$ is of the form $$\sigma_{h,e,b}(u)=\sum_{i=0}^{n-1}a_{i}b^{i}$$ with $a_{0},...,a_{\delta}=b-1$. More informally, the rightmost $\delta + 1$ digits of $\sigma_{h,e,b}(u)$ are all $b-1$.
\end{thm}
\begin{proof}
In this proof, we will show that if $\sigma_{h,e,b}$ has ``too many'' digits which are not equal to $b-1$, we can construct a smaller number with the same image as $\sigma_{h,e,b}$. This contradicts the definition of $\sigma_{h,e,b}$.

One can verify $\sigma_{1,e,b}(1)=10$ (in base $b$) for all $e,b$ and that this is the only number of the form $\sigma_{h,e,b}$ with a 0 digit. However, $10$ is a 2 digit number and $d > 2$ for integers $e>1$. Thus, using the base $b$ expansion from the statement of the theorem, $a_{i+1} \leq a_{i}$ for $0 \leq i < n-1$ (its digits must appear in increasing order from left to right) and none of its digits can be 0 since  $\sigma_{h,e,b}(u)$ is the least height $h$, $u$ attracted number. 

In the case $a_i=b-1$ for all $i$, this theorem is trivially true. Otherwise, let us construct $z$, the sum of the image of the digits which are not equal to $b-1$. In the case that some digits of $\sigma_{h,e,b,}(u)$ are $b-1$ and some are not, define an integer parameter $k \geq 2$ to be such that $a_{k-1}<b-1$ and for all $i< k-1$, $a_{i}=b-1$. That is, the $k^\text{th}$ place is the first (from the right) in which a digit that is not $b-1$ appears. Hence, $$\sigma_{h,e,b}(u)=\sum_{i=k-1}^{n-1}a_{i}b^{i}+\sum_{i=0}^{k-2}(b-1)b^{i}.$$ Let $y=S_{e,b}(\sigma_{h,e,b}(u))$ and let $z=y-(k-1)(b-1)^{e}$, that is, $$z=\sum_{i=k-1}^{n-1}a_{i}^e.$$
In the case that no digits of $\sigma_{h,e,b}$ are $b-1$, set $k=1$ and let $z=\sum_{i=0}^{n-1}a_{i}^e.$ We proceed to show that if $k\leq \delta+1$, we can construct a number smaller than $\sigma_{h,e,b}$ with the same image as $\sigma_{h,e,b}$, a contradiction. Let $n'=n-(k-1)$ and let $m=\lfloor \frac{z}{(b-1)^{e}}\rfloor$.  
Since $z$ is the sum of $n'$ many terms of the form $a_{i}^{e}$ where $a_{i}\leq b-2$ for all $i$,  $n'\geq \frac{z}{(b-2)^{e}}$. Thus,  $\frac{(b-2)^{e}}{(b-1)^{e}}n' \geq  \frac{z}{(b-1)^{e}} \geq m$.

So, $$\left(\frac{b-2}{b-1}\right)^{e}n'+g(e)+1\geq m+g(e)+1.$$ 
By the definition of $d$, $d-\delta=\frac{g(e)+1}{1-(\frac{b-2}{b-1})^{e}}$, and since $k\leq \delta+1$, $d-(k-1)\geq \frac{g(e)+1}{1-(\frac{b-2}{b-1})^{e}}$. Thus,
$$(d-(k-1))\left(1-\left(\frac{b-2}{b-1}\right)^{e}\right)  \geq  g(e)+1.$$
And since $n'\geq d-(k-1)$ and $1-(\frac{b-2}{b-1})^{e}>0$, we have that $n'(1-(\frac{b-2}{b-1})^{e}) \geq  g(e)+1$ and hence
$$n' \geq  g(e)+1+n'\left(\frac{b-2}{b-1}\right)^{e}  \geq  m+g(e) + 1.$$
Therefore, $n'\geq m + g(e)+1$.

Let $r$ be the remainder of $y$ divided by $(b-1)^{e}$, that is $y=q(b-1)^e+r$ where $q \geq 0$ and $(b-1)^e>r \geq 0$. From the definition of $m$, $q=m+(k-1)$. Let $x_{1},x_{2},...,x_{g(e)}$ be integers less than $b-1$ so that $x_{1}^{e}+x_{2}^{e}+...+x_{g(e)}^{e}=r$. There are such $x_{j}$ since $g(e)$ is defined so that such integers exist, and all integers must be less than $b-1$ since $r<(b-1)^{e}$. Without loss of generality, $x_{1}\leq x_{2} \leq ... \leq x_{g(e)}$. Let $x$ be a base $b$ number with digits $x_1, ..., x_{g(e)}$ followed by $m+(k-1)$ many $b-1$ digits.

Hence, $S_{e,b}(x)=y$, and $x$ has at most $g(e)+m+(k-1)$ digits.
Since $n'=n-(k-1)$, $n\geq g(e)+1+m+(k-1)$. However, this means that $x$ has fewer digits than $\sigma_{h,e,b}(u)$. This contradicts the fact that $\sigma_{h,e,b}(u)$ is the smallest height $h$, $u$ attracted integer, and hence, $k>\delta+1$.
\end{proof}

For ease of notation, we define a constant $p_{e,b}$.

\begin{defn}
For a fixed exponent $e$ and base $b$, let $p_{e,b}$ be the smallest integer such that $b^{p_{e,b}}>g(e)$.
\end{defn}

\begin{thm}
\label{L:ninesmap}
Fix a base $b$, exponent $e$, positive height $h$, and attractor $u$. If $\sigma_{h,e,b}(u)=\sum_{i=0}^{n-1}a_{i}b^i$ where $a_{0},...,a_{e+p_{e,b}}=b-1$, then $S_{e,b}(\sigma_{h+1,e,b}(u))=\sigma_{h,e,b}(u)$.
\end{thm}
\begin{proof}
Let $\sigma_{h,e,b}(u)$ be such that $a_{0},...,a_{k}=b-1$ where $k\geq e+p_{e,b}$. Define $c_{j}=\sigma_{h,e,b}(u)+j$ for $1\leq j < g(e)(b-1)^{e}$. We will show that $c_{1}$ through $c_{g(e)(b-1)^{e}-1}$ are not height $h$, $u$ attracted numbers.

If $b >2$, using the definition of $p_{e,b}$ we get
$$j < g(e)(b-1)^{e} <  b^{p_{e,b}}(b-1)^{e} <  b^{p_{e,b}} b^{e} =  b^{e+p_{e,b}}.$$ Since $\sigma_{h,e,b}$ has at least $e+p_{e,b}+1$ trailing digits equal to $b-1$, $c_{1}$ has at least $e+p_{e,b}+1$ trailing zeros. Since $j<b^{e+p_{e,b}}$, $j$ has at most $e+p_{e,b}$ many digits. Hence $c_j$ has at least one digits which is zero for $1\leq j < g(e)(b-1)^{e}$.  Let $c_{j}'$ be formed by removing the all zero digits of $c_{j}$. We claim that $c_{j}'<\sigma_{h,e,b}(u)$. Recall that $n$ denotes the number of digits of $\sigma_{h,e,b}(u)$. If $a_{i}\neq b-1$ for some $i$, then $n \geq e+p_{e,b}+2$ and $c_j$ has $n$ digits for all $j$. Thus, $c_{j}'$ has at most $n-1$ digits and hence $c_j'<\sigma_{h,e,b}$.  If $a_{i}=b-1$ for all $i$, then $\sigma_{h,e,b}(u)=b^n-1$ and $c_1=b^n=b^{e+p_{e,b}+1}$, which means that $c_j< b^{e+p_{e,b}+1}+b^{e+p_{e,b}}$. Thus $c_{j}'$ has at most $n$ digits, while the leading digit of $\sigma_{h,e,b}$ is $b-1$, but the leading digit of $c_{j}'$ is 1, and since $b\neq 2$, $c_j'<\sigma_{h,e,b}$. %If $b=2$, then note that $(b-1)^e=1$, so we can instead add up to $b^{p_{e,b}}$ to $\sigma_{h,e,b}(u)$, so $c_j < b^{e+p_{e,b}+1}+b^{p_{e,b}}$. Since $e$ is at least 1, $c_j$ has at least 2 zero digits, 

This leaves only the case that $b=2$. In this case,
$$j < g(e)(2-1)^{e} = g(e) < 2^{p_{e,2}}.$$
Since the only allowable digits are 0 and 1, and we argued in the proof of Theorem~\ref{L:trailnines} that $\sigma_{h,e,b}$ does not have any digits that are equal to zero, $\sigma_{h,e,2}=2^{n+1}-1$ for some $n \geq e+p_{e,2}$, so $2^{n+1} \leq c_j < 2^{n+1} + 2^{p_{e,2}}$ for all $j$. Since $n \geq e+p_{e,2}$ and $e$ is at least 1, $c_j$ has at least 2 digits, that are equal to 0. Again, let $c_{j}'$ be formed by removing the all zero digits of $c_{j}$. Then $c_j'$ has fewer than $n$ digits and hence $c_j'<\sigma_{h,e,2}$.

So, if any $c_{j}$ are height $h$, $u$-attracted numbers, then $c_{j}'$ is a smaller height $h$, $u$ attracted number than $\sigma_{h,e,b}(u)$, contradicting the definition of $\sigma_{h,e,b}(u)$. Hence, $\tau_{h,e,b}(u) \geq g(e)(b-1)^{e}+\sigma_{h,e,b}(u)$. Therefore, by Theorem~\ref{L:willmap}, $S_{e,b}(\sigma_{h+1,e,b})=\sigma_{h,e,b}$.
\end{proof}

\begin{cor}\label{C:upperbnd}
Fix a base $b$ and exponent $e$. Let $d=\lceil \frac{g(e)+1}{1-(\frac{b-2}{b-1})^{e}}+e+p_{e,b}\rceil$. If $\sigma_{h,e,b}(u)\geq b^{d}$, then $S_{e,b}(\sigma_{h+1,e,b}(u))=\sigma_{h,e,b}(u)$.
\end{cor}
\begin{proof}
Since $\sigma_{h,e,b}(u)\geq b^{d}$, $\sigma_{h,e,b}(u)$ must have at least $d-1$ digits. Hence, by Theorem~\ref{L:trailnines}, $\sigma_{h,e,b}(u)=\sum_{i=0}^{n-1}a_{i}b^{i}$ where for $i\leq e+p_{e,b}$, $a_{i}=b-1$. Therefore, by Theorem~\ref{L:ninesmap}, $S_{e,b}(\sigma_{h+1,e,b}(u))=\sigma_{h,e,b}(u)$.
\end{proof}

Corollary~\ref{C:upperbnd} gives a bound $b^d$ for $\sigma_{h,e,b}(u)$ (in terms of $e$ and $b$) so that if $\sigma_{h,e,b}(u)\geq b^{d}$, then $S_{e,b}(\sigma_{h+1,e,b}(u))=\sigma_{h,e,b}$. This leads to the natural question:
\begin{que}
Is there a bound $\beta$ for $h$ (in terms of $e$ and $b$) so that if $h \geq \beta$ $S_{e,b}(\sigma_{h+1,e,b}(u))=\sigma_{h,e,b}$?
\end{que}

\section*{Acknowledgements}
This work was supported by a Bowen Summer Research Assistantship from Denison University. The authors also thank the referee for their helpful suggestions. Finally, the authors would like to acknowledge the Research Experiences for Undergraduate Faculty program.

\end{document}